\documentclass[11pt]{article}
\usepackage{latexsym, amssymb, enumerate, amsmath, amsthm,pdfsync, multicol,numbysec}
\usepackage{fullpage}
\usepackage{layout}
\usepackage{cleveref}
\usepackage{cite}
\usepackage{url}
\usepackage{esint}
\usepackage{dsfont}
\usepackage{color}
\usepackage{overpic}

\usepackage{graphicx,epsfig}

\crefname{assumption}{assumption}{assumptions}
\crefname{thm}{theorem}{theorems}
\crefname{lem}{lemma}{lemmas}
\crefname{cor}{corollary}{corollaries}
\crefname{prop}{proposition}{propositions}
\Crefname{theorem}{Theorem}{Theorems}
\crefname{conjecture}{conjecture}{conjectures}

\newcommand{\cE}{\mathcal{E}}
\newcommand{\cA}{\mathcal{A}}

\newcommand{\R}{\mathbb{R}}

\newcommand{\eps}{\varepsilon}

\newcommand{\1}{\mathds{1}}
\let\phi\varphi

\DeclareMathOperator{\dist}{dist}

\newtheorem{thm}{Theorem}[section]
\newtheorem{prop}[thm]{Proposition}
\newtheorem{lem}[thm]{Lemma}

\newtheorem{conjecture}[thm]{Conjecture}

\newtheorem{rem}[thm]{Remark}

\date\today
\author{Emeric Bouin 
\footnote{CEREMADE - Universit\'e Paris-Dauphine, UMR CNRS 7534, Place du Mar\'echal de Lattre de Tassigny, 75775 Paris Cedex 16, France. E-mail: \texttt{bouin@ceremade.dauphine.fr}}\and
Christopher Henderson \footnote{Labex MILYON \& UMPA \& INRIA NUMED Team, Ecole Normale Sup\'erieure de Lyon, UMR CNRS 5669 'UMPA', 46 all\'ee d'Italie, F-69364~Lyon~cedex~07, France. E-mail: \texttt{christopher.henderson@ens-lyon.fr}}}

\numberbysection

\begin{document}

\title{Super-linear spreading in local bistable cane toads equations}
\maketitle

\begin{abstract}
In this paper, we study the influence of an Allee effect on the spreading rate in a local reaction-diffusion-mutation equation modeling the invasion of cane toads in Australia. %The population of toads is structured by a phenotypical trait that governs the spatial diffusion.
We are, in particular, concerned with the case when the diffusivity can take unbounded values.  We show that the acceleration feature that arises in this model with a Fisher-KPP, or monostable, non-linearity still occurs when this non-linearity is instead bistable, despite the fact that this kills the small populations.  This is in stark contrast to the work of Alfaro, Gui-Huan, and Mellet-Roquejoffre-Sire in related models, where the change to a bistable non-linearity prevents acceleration.
\end{abstract}
%\section{}
%\subsection{}
\noindent{\bf Key-Words:} {Structured populations, reaction-diffusion equations, front acceleration}\\
\noindent{\bf AMS Class. No:} {35Q92, 45K05, 35C07}

\section{Introduction}

\subsection{The biological model}

The invasion of cane toads in Australia has interesting features different from the standard spreading observed in 
most other
species. The experimental data~\cite{PhillipsBrownEtAl, Shine} show that the invasion speed has steadily increased 
during 
the eighty years since the toads were introduced in Australia.  In addition, the younger individuals at 
the edge of 
the invasion front have a significantly different morphology compared to other populations -- their legs 
tend to be on average much longer than away from the front. 
This is just one example of  a non-uniform space-trait 
distribution -- see, for instance, a study on the expansion of bush crickets in Britain~\cite{Thomas}.  Several works have addressed the  front 
invasions in ecology, where the trait is related to the
dispersal ability \cite{ArnoldDesvillettesPrevost,ChampagnatMeleard}. It has been observed
that selection of more mobile individuals can occur, even if they have no advantage 
in their reproductive rate, due to the spatial sorting 
\cite{Kokko,Ronce,PhillipsBrownEtAl,Simmons}.

In this paper, we focus on a model for this phenomenon (proposed in~\cite{BenichouEtAl},
based on the classical Fisher-KPP equation~\cite{Fisher,KPP} ).  The population density is structured by a spatial variable, $x\in \R$, and a 
motility variable $\theta\in \Theta\stackrel{\rm def}{=}[\underline\theta,\infty)$, with a  fixed   $\underline\theta>0$. 
This population undergoes diffusion in the trait variable $\theta$, with a constant diffusion coefficient $\alpha>0$, representing 
mutation, and in the spatial variable, with the diffusion coefficient~$\theta$, representing the effect of the trait  on the spatial spreading 
rates of the species. 

In the model proposed in~\cite{BenichouEtAl}, the reaction is of non-local, monostable type.  In the case where $\Theta$ is bounded, the problem has been shown to be well-posed~\cite{Turanova} and exhibit travelling-wave solutions~\cite{BouinCalvez}.  When $\Theta$ is unbounded, sharp results have been proven for the local analogue while weaker results have been obtained for the non-local equation~\cite{BerestyckiMouhotRaoul,BHR_acceleration}.  In particular, the Cauchy problem exhibits propagation of of the order $t^{3/2}$.

\subsection{Motivation}

With a monostable non-linearity, acceleration, i.e.~super-linear spreading of level sets, has been shown in various settings in mathematical biology~\cite{BBHHR, BouinCalvezNadin, CabreRoquejoffre, CCR, CoulonRoquejoffre, Garnier, HamelRoques, HarrisHarris, HendersonFast, MeleardMirrahimi, RoquejoffreTarfulea}.  We mention, in particular, the existence of acceleration for Fisher-KPP with fat-tailed initial data and for Fisher-KPP with the fractional Laplacian.  In these models, acceleration is related to the notion of pulled fronts.  In particular, we see that small populations far from the origin grow exponentially, causing the acceleration.  Recently, in these settings, it has been shown that when the non-linearity is ignition or bistable type, there is no acceleration~\cite{Alfaro, GuiHuan,MelletRoquejoffreSire}.  Indeed, the small populations that drive the acceleration are killed when the non-linearity is bistable.  In other words, acceleration is a tail phenomenon, so when a bistable non-linearity kills the tails, there is no acceleration.

In this article, we obtain results in stark contrast with the results discussed above.  Indeed, we demonstrate propagation of the order $t^{3/2}$ in the local cane toads equation with a bistable non-linearity.  This shows that, for the cane toads equation, acceleration is a bulk phenomenon, i.e. that it is not driven by small populations far from the origin.

\subsection{Main results}

Fix a bistable non-linearity $f: [0,1] \to [0,\infty]$ that is Lipschitz continuous function such that there exists $\alpha \in (0,\frac12)$ with
\[
	f(u) \geq u(u - \alpha)(1 - u).
\]
We assume also that $f(0) = f(1) = 0$.  We are interested in the long-time asymptotics of solutions to the Cauchy problem
\begin{equation}\label{eq:bistable}
	\begin{cases}
		u_t = \theta u_{xx} + u_{\theta\theta} + f(u), \qquad &(t,x,\theta) \in \R^+\times \R\times \Theta \medskip,\\
		u_\theta(t,x,\underline \theta) = 0, 	\qquad &(t,x)\in \R^+ \times \R, \medskip\\
		u(0,x,\theta) = u_0(x,\theta) &(x,\theta) \in \R \times \Theta.\\  
	\end{cases}
\end{equation}
The initial data function $u_0$ is assumed to satisfy
\begin{equation*}
u_0\geq \1_{\R^- \times[\underline\theta, (1+\lambda)\underline\theta]},
\end{equation*}
where $\lambda > 0$.  

Our interest is in understanding where the ``front'' of $u$ is.  In other words, we will fix a level set and understand its propagation.  Using the same equation with a monostable non-linearity as a super-solution to~\eqref{eq:bistable}, the results from~\cite{BerestyckiMouhotRaoul, BHR_acceleration} show that no level set can move faster than $O(t^{3/2})$.  Our main result is to show that a lower bound of the same order holds as well.

\begin{thm}\label{thm:acceleration}
	Fix any $m\in (0,1)$.  There exists $\lambda_0 \in \R_+^*$, depending only on $\alpha$ and $\underline\theta$, and $\gamma \in \R_+^*$, depending only on $\alpha$, such that if $\lambda \geq \lambda_0$ then
	\[
		\liminf_{t\to\infty} \; \frac{\max\{ x :  \exists \theta \in \Theta, \; u(t,x,\theta) = m\}}{t^{3/2}} \; \geq \gamma.
	\]
\end{thm}

Before we continue, we note that we expect that this theorem holds in much greater generality.  However, our interest is in showing that the acceleration in the cane toads equation is a also bulk phenomenon, as opposed to a tail phenomenon.  Hence, we seek to simply show that acceleration occurs despite the bistable non-linearity, and we do not attempt to find the most general setting or the sharpest bounds.

To prove this theorem, we will follow a similar strategy as in a previous paper by the authors and Ryzhik~\cite{BHR_acceleration} and that also appears in~\cite{BerestyckiHamelNadin,RR-Toulouse}. In rough words, we will slide a suitable ``bump" along suitable trajectories in the phase plane $\R \times \Theta$, making sure that the ball remains below the solution of the original Cauchy problem \eqref{eq:bistable}. Once the trajectories are well chosen, this will imply the acceleration phenomena claimed in \cref{thm:acceleration}.  However, the ``bump'' sub-solution is significantly more complicated to create in this setting than in the monostable one.  Indeed, in~\cite{BHR_acceleration}, the sub-solution is created almost entirely with the linearized (around zero) problem.  In our setting, of course, the linearized problem decays to zero on any traveling ball.  We describe how to overcome this difficulty below.  This is the main technical difficulty in the present article.

\subsection{Strategy of the proof: constructing the sub-solution}

We will now end this introduction by describing and presenting more precisely the major objects that are needed to prove \cref{thm:acceleration}, namely the suitable trajectory and ``bump". To begin, we fix a large $T> 0$ and any level set height $m \in (0,1)$. We then define a trajectory in the space $\R \times \Theta$ by 
\begin{equation*}
t \mapsto (X_T(t),\Theta_T(t)),
\end{equation*}
for any $t \in [0,T]$.  The coordinate functions $(X_T,\Theta_T)$ will be defined later on (see \cref{sec:proof}), and are crucial for our analysis. In order to slide a bump along this trajectory, we define the moving (growing) ellipse
\begin{equation}\label{eq:ellipse}
	\mathcal{E}_{t,\Lambda} = \left\{ (x,\theta) \in \R\times \Theta : \frac{(x - X_T(t))^2}{\Theta_T(t)} + (\theta - \Theta_T(t))^2 \leq \Lambda^2 \right\}
\end{equation}
and the moving (growing) annulus
\begin{equation}\label{eq:annulus}
	\mathcal{A}_{t,\Lambda} =  \left\{ (x,\theta) \in \R\times \Theta : \Lambda^2 \leq \frac{(x -X_T(t))^2}{\Theta_T(t)} + (\theta - \Theta_T(t))^2 \leq 4\Lambda^2 \right\},
\end{equation}
where $\Lambda$ is a positive constant to be chosen later that encodes the sizes of this two objects.  For now, our only assumption on $\Lambda$ is that $\Lambda \leq \lambda \underline \theta / 8$.  Our goal is to build a sub-solution to~\eqref{eq:bistable} on the bigger ellipse $\mathcal{A}_{t,\Lambda}\cup \mathcal{E}_{t,\Lambda}$, what we called a ``bump'' above.

Since it is the main issue of the paper, let us now explain how we build a sub-solution on $\mathcal{A}_{t,\Lambda}\cup \mathcal{E}_{t,\Lambda}$. We shall patch together a solution $w^+ \geq \alpha$ on $\mathcal{E}_{t,\Lambda}$ of
\begin{equation}\label{eq:w+}
	\begin{cases}
w^+_t = \theta w^+_{xx} + w^+_{\theta\theta} + f_r(w^+), \qquad &(x,\theta) \in \mathcal{E}_{t,\Lambda}, \medskip\\
w^+ \equiv \alpha, \qquad &(x,\theta) \in \partial \mathcal{E}_{t,\Lambda},
	\end{cases}
\end{equation}
and a positive solution $w^- \leq \alpha$ on $\mathcal{A}_{t,\Lambda}$ of
\begin{equation}\label{eq:w-}
	\begin{cases}
		w^-_t = \theta w^-_{xx} + w^-_{\theta\theta} + f_r(w^-), \qquad & (x,\theta) \in \mathcal{A}_{t,\Lambda}, \medskip\\
w^- \equiv \alpha, & (x,\theta) \in \partial \mathcal{E}_{t,\Lambda}, \medskip\\
		w^- \equiv 0,
			& (x,\theta) \in \partial \mathcal{A}_{t,\Lambda} \setminus \partial \mathcal{E}_{t,\Lambda}.
	\end{cases}
\end{equation}

Notice that in these two previous definitions, there is a small but important discrepancy between the nonlinearities used and in \eqref{eq:bistable}. The new non-linearity $f_r$ is defined as follows for any $r \in (2\alpha, 1]$:
\begin{equation}\label{eq:f_r}
	f_r(u) \stackrel{\rm def}{=}
		\begin{cases}
			u( u - \alpha)(1 - u), ~~~~&\text{ for } u \leq \alpha,\medskip\\
			c_r u( u - \alpha)(r - u), &\text{ for } u \geq \alpha.
		\end{cases}
\end{equation}
where we define $c_r = (1-\alpha)(r-\alpha)^{-1}$.  It is easily verified that $f_r(u) \leq f(u)$ for all $u$, and it is clear that $f_r$ is Lipschitz continuous.  For technical reasons, we are required to take $r$ to be slightly less than $1$ in the sequel.

%\begin{overpic}%[grid,tics=10]
%	{/Users/chen/Dropbox/Math/Talks/Pictures/domain.png}
%	 \put (35,80){$\Omega$}
%	 \put (40.5,41) {$u_0\equiv 1$}
%	 \put (60,65) {$u_0\equiv 0$}
%	 \put (48,55) {$\uparrow$}
%	 \put (60,41) {$\rightarrow$}
%	 \put (30,41) {$\leftarrow$}
%	 \put (48, 28) {$\downarrow$}
%%	 \put (50,50) {$\nearrow$}
%%	 \put (50,50) {$\nwarrow$}
%%	 \put (50,50) {$\swarrow$}
%%	 \put (50,50) {$\searrow$}
%	\end{overpic}

\begin{figure}[h]
\begin{center}
\begin{overpic}%[grid,tics=10]
		{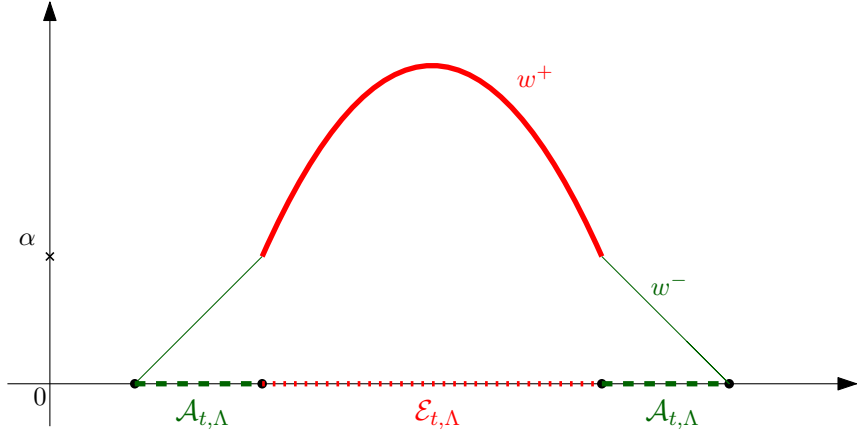}
	\put (75,1){{\color[rgb]{0,.4,0}$\mathcal{A}_{t,\Lambda}$}}
	\put (20,1){{\color[rgb]{0,.4,0}$\mathcal{A}_{t,\Lambda}$}}
	\put (48,1){{\color{red}$\mathcal{E}_{t,\Lambda}$}}
\end{overpic}
\caption{Schematic plot of the flying saucer like sub-solution $w$ that is slid along the trajectory. The bolded green line denotes the set $\mathcal{A}_t$, the red dotted one the set $\mathcal{E}_t$.}
\end{center}
\end{figure}
%\begin{figure}[h]
%\begin{center}
%\includegraphics[width=.6\linewidth]{W.eps}
%\end{center}
%\caption{Schematic plot of the flying saucer like sub-solution $w$ that is slid along the trajectory. The bolded green line denotes the set $\mathcal{A}_t$, the red dotted one the set $\mathcal{E}_t$.}
%\end{figure}
Now that we have defined $w^{\pm}$, we obtain a sub-solution $w$ of \eqref{eq:bistable} on $\R^+ \times \R \times \Theta$ defined by
\begin{equation}\label{eq:w}
	w(t,x,\theta) = w^+ \1_{\mathcal{E}_{t,\Lambda}} + w^- \1_{\mathcal{A}_{t,\Lambda}}.
\end{equation}
By construction, $w$ is Lipschitz continuous. However, it need not be a $C^1$ function along $\partial \mathcal{A}_{t,\Lambda} \setminus \partial \mathcal{E}_{t,\Lambda}$ and $\partial \mathcal{E}_{t,\Lambda}$.  In order to be a sub-solution, we must check the convexity of $w$ along both boundaries.  The positivity of $w^-$ ensures that the convexity of $w$ is correct at the boundary $\partial \mathcal{A}_{t,\Lambda} \setminus \partial \mathcal{E}_{t,\Lambda}$. However, to make sure that it is indeed a sub-solution at the boundary $\partial \mathcal{E}_{t,\Lambda}$ we need to check properly that the normal derivatives along this boundary are well ordered:
%
%
% Hence though it is a sub-solution to $u$ on each of $A_t$ and $\mathcal{E}_{t,\Lambda}$ separately, it is not necessarily a sub-solution on $A_t \cup \mathcal{E}_{t,\Lambda}$.  On the other hand, if we show that
\begin{equation}\label{eq:good_convexity}
	|\partial_n w^+| \geq |\partial_n w^-| ~~~~ \text{ on } \partial \mathcal{E}_{t,\Lambda},
\end{equation}
where $\partial_n$ is the (outward) normal derivative to the boundary of $\mathcal{E}_{t,\Lambda}$. This is the main technical issue at hand in the proof.

\subsection{Numerics and comments}

\subsection*{Numerics}

Let us enlighten \cref{thm:acceleration} by showing in \cref{fig:Shape} some numerical insights of the acceleration behavior. One can compare these to \cite{BCMetal,BerestyckiMouhotRaoul} that are numerics related to the study of the acceleration in the nonlocal cane toads equation.

\begin{figure}
\begin{center}
\includegraphics[width = .45\linewidth]{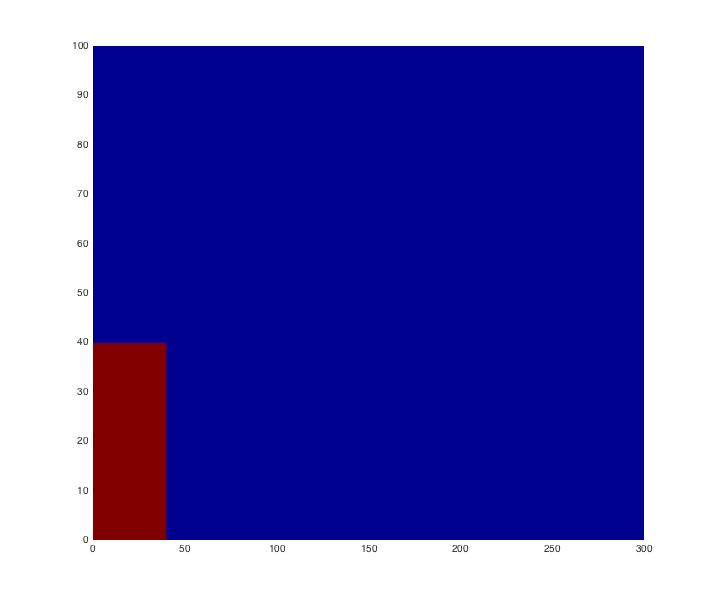} \; 
\includegraphics[width = .45\linewidth]{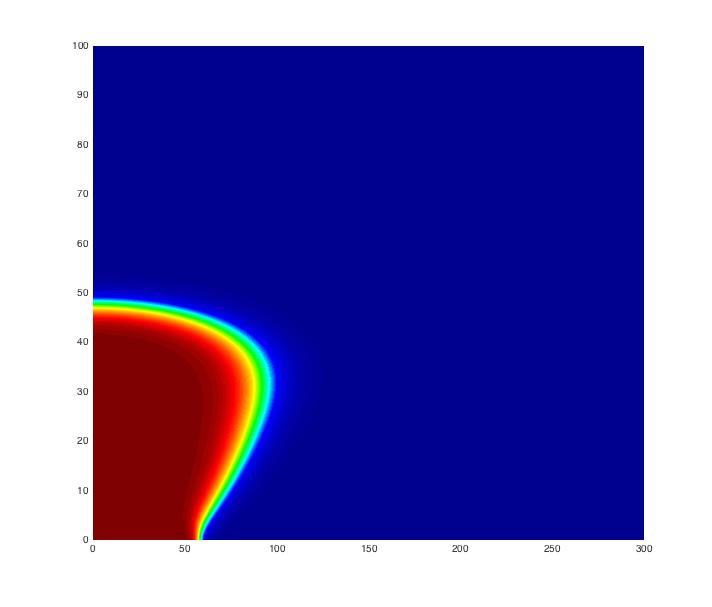}
\includegraphics[width = .45\linewidth]{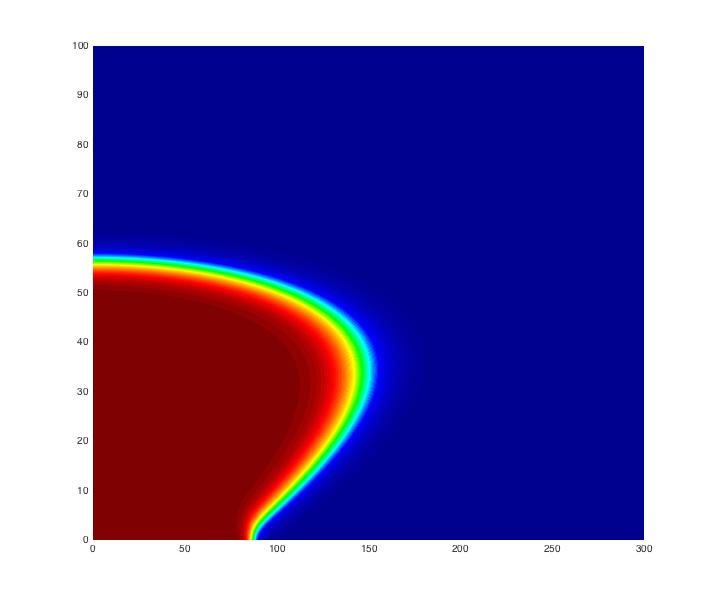} \; 
\includegraphics[width = .45\linewidth]{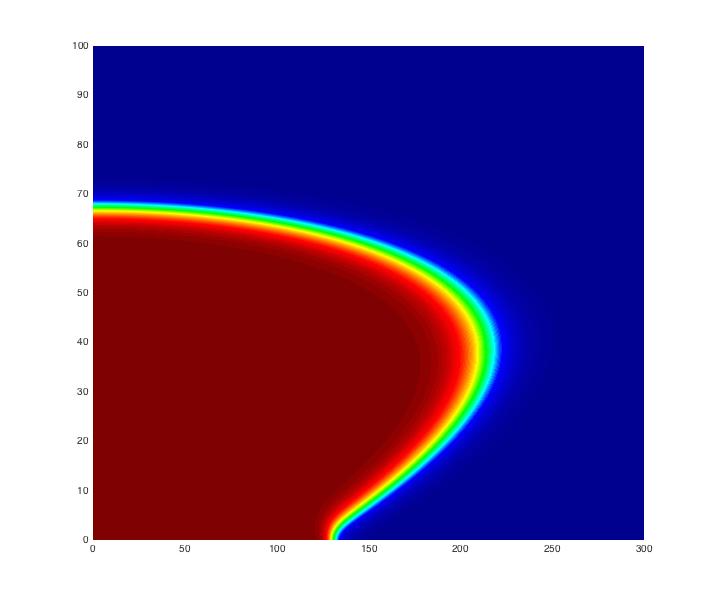} \;
\includegraphics[width = .45\linewidth]{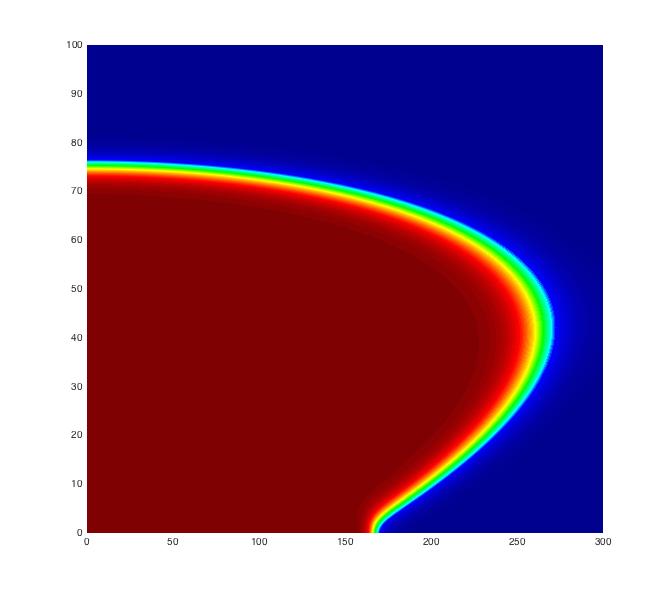}
\caption{Numerical simulations of the Cauchy problem of equation \eqref{eq:bistable} at a fixed time, in the phase space $\R\times\Theta$. From top left to bottom right: \textit{the initial data}, $t=30$, $t=60$, $t=90$, $t=120$. One can track the accelerated behavior, for example on the space axis. The invasion at the back in the $\theta$-direction is expected to be linear in time. This pattern is very similar as for the monostable cane toads equation, see also~\cite{BCMetal,BerestyckiMouhotRaoul,BHR_acceleration}.}
\label{fig:Shape}
\end{center}
\end{figure}

\subsection*{A nonlocal bistable model}

We conclude the introduction by discussing a nonlocal version of \eqref{eq:bistable}: 
\begin{equation}\label{eq:nlbistable}
	\begin{cases}
		n_t = \theta n_{xx} + n_{\theta\theta} + n(\rho - \alpha)(1-\rho), \qquad &(t,x,\theta) \in \R^+\times \R\times \Theta \medskip,\\
		\rho(t,x) := \displaystyle \int_{\underline \theta}^{\infty} n(t,x,\theta) \, d\theta, \qquad &(t,x) \in \R^+\times \R \medskip,\\
		n_\theta(t,x,\underline \theta) = 0, 	\qquad &(t,x)\in \R^+ \times \R, \medskip\\
		n(0,x,\theta) = n_0(x,\theta) &(x,\theta) \in \R \times \Theta.\\  
	\end{cases}
\end{equation}

It is easy to see that one many bound the propagation rate above by using the monostable model with growth rate $r = \sup_{\rho \in \R^+} \left( (\rho - \alpha)(1-\rho) \right)$.  This implies that the propagation can be no faster than $O(t^{3/2})$.  However, obtaining a lower bound is significantly more complicated.  We are led to the following conjecture, which we are unable to prove at this time.

\begin{conjecture}
	The model~\eqref{eq:nlbistable} exhibits acceleration.  That is, the level sets of $\rho$ move super-linearly in time.
\end{conjecture}
%We conjecture that an accelerated spreading rate occurs (and thus that a accelerating lower bound holds also) and
We provide some numerics in \cref{fig:nlshape} to support this conjecture. However, the proof of such a result is far beyond the scope of this paper. Indeed, it requires to study the dynamics of the zone where $\rho$ is greater that $\alpha$, which is not at the edge of the front at all, so that an argument by contradiction as in \cite{BHR_acceleration} is impossible. Moreover, to our knowledge, it is not possible to derive a probabilistic framework such as in \cite{BerestyckiMouhotRaoul} (and the references therein) to be able to study the nonlocal model~\eqref{eq:nlbistable}. %Some new tools might be needed to study \eqref{eq:nlbistable},
As new tools will be needed to study this model, we thus leave this conjecture for further investigation.  

\begin{figure}
\begin{center}
\includegraphics[width = .45\linewidth]{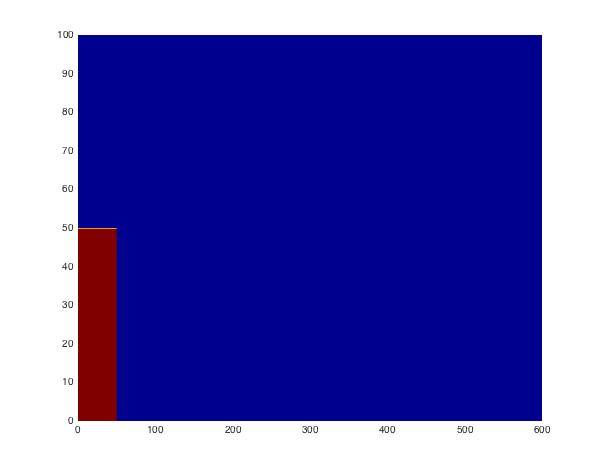} \; 
\includegraphics[width = .45\linewidth]{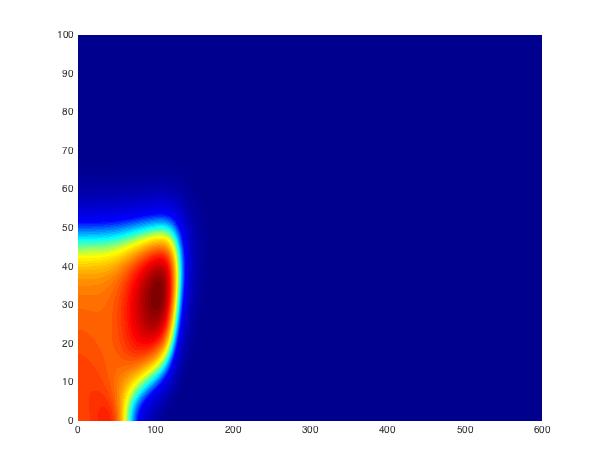}
\includegraphics[width = .45\linewidth]{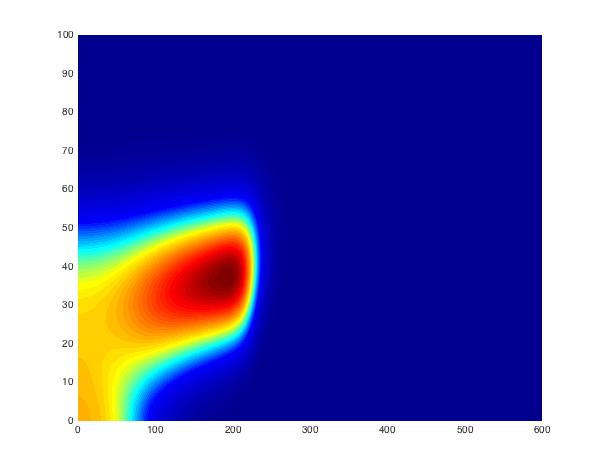} \; 
\includegraphics[width = .45\linewidth]{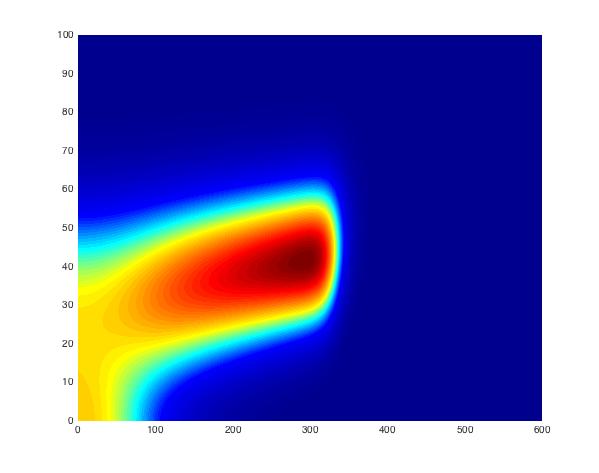} \;
\includegraphics[width = .45\linewidth]{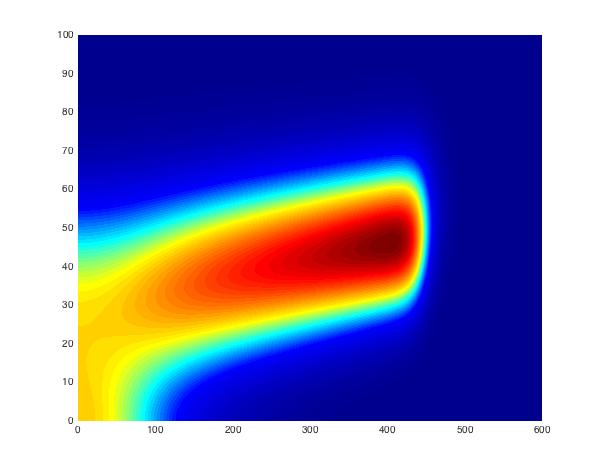}
\caption{Numerical simulations of the Cauchy problem of equation \eqref{eq:nlbistable} at a fixed time, in the phase space $\R\times\Theta$. From top left to bottom right: \textit{the initial data}, $t=30$, $t=60$, $t=90$, $t=120$. Here we see the distinctive "up-and-over" behavior that leads to acceleration for the non-local, monostable model.  One can track the accelerated behavior, for example on the space axis.}
\label{fig:nlshape}
\end{center}
\end{figure}

\subsection*{Outline of the article}
In \cref{sec:proof}, we give the main outline of the proof of \cref{thm:acceleration} assuming the main \cref{prop:enough_mass}. In \cref{sec:lemma}, we prove this proposition by finishing the construction of the subsolution \eqref{eq:w}, and reducing it to showing a \cref{prop:subsol} that gives \eqref{eq:good_convexity}. This main lemma involves understanding a related steady state problem, which we analyze in \cref{sec:steady,sec:convergence_steady}.

\subsection*{Acknowledgments}  EB is very grateful to Cambridge University for its hospitality during the second semester of the academic year 2015-2016. EB and CH acknowledge the support of the ERC Grant MATKIT (ERC-2011-StG).  Part of this work was performed within the framework of the LABEX MILYON (ANR-10-LABX-0070) of Universit\'e de Lyon, within the program ``Investissements d'Avenir" (ANR-11-IDEX-0007) operated by the French National Research Agency (ANR).  In addition, this project has received funding from the European Research Council (ERC) under the European Unions Horizon 2020 research and innovation programme
(grant agreement No 639638).

\section{The acceleration feature - proof of \cref{thm:acceleration}}\label{sec:proof}

In this section, we prove \cref{thm:acceleration}. For this purpose, we first define carefully the trajectory that has been introduced in the previous section. Second, we take for granted the existence of a sub-solution of the form \eqref{eq:w} (this will the object of \cref{sec:lemma}) and we show how this implies the acceleration feature stated in \cref{thm:acceleration}.

% 
%The proof proceeds in two steps: first, we build a sub-solution on a ball that slides at a constant rate purely in the positive $\theta$ direction, and second, we build a sub-solution on a ball that slides at a constant rate purely in the positive $x$ direction.  In the second step, the ball may slide at a rate proportional to the square root of its height, i.e. minimum $\theta$ value.  Here is where we see the acceleration.
%
We now define the trajectory $t \in [0,T] \mapsto (X_T(t),\Theta_T(t)) \in \R \times \Theta$. For this purpose, we concatenate two trajectories $t \mapsto (X_T^{(1)}(t),\Theta_T^{(1)}(t))$ and $t \mapsto (X_T^{(2)}(t),\Theta_T^{(2)}(t))$:
\begin{equation*}
X_T = 
\begin{cases}
%\begin{array}{l}
X_T^{(1)}, & t \in \left[0,\frac{T}{2}\right],\medskip\\
X_T^{(2)}, & t \in \left[\frac{T}{2},T\right],\\
\end{cases}
\qquad \Theta_T = 
\begin{cases}
%\begin{array}{l}
\Theta_T^{(1)}, & t \in [0,\frac{T}{2}],\medskip\\
\Theta_T^{(2)}, & t \in [\frac{T}{2},T].\\
\end{cases}
\end{equation*}
These latter trajectories are defined by
\begin{equation*}
\forall t \in \left[0,\frac{T}{2} \right],
	\qquad \left(X_T^{(1)},\Theta_T^{(1)}\right)(t)
	= \left(  - \frac{\lambda}{4} \underline \theta ,~ c t + \left( 1 + \frac{3\lambda}{4} \right) \underline\theta\right),
\end{equation*}
and
\begin{equation*}
\forall t \in \left[\frac{T}{2},T \right], \qquad 
	\left(X_T^{(2)},\Theta_T^{(2)}\right)(t)
	= \left( \frac{c  \left(t - \frac{T}{2} \right)T}{\sqrt{\frac{c T}{2} + \left( 1 + \frac{3\lambda}{4} \right) \underline\theta}} - \frac{\lambda}{4} \underline \theta , \frac{c T}{2} + \left( 1 + \frac{3\lambda}{4} \right) \underline\theta \right). 
\end{equation*}

\begin{figure}[h]
\begin{center}
\includegraphics[width=.6\linewidth]{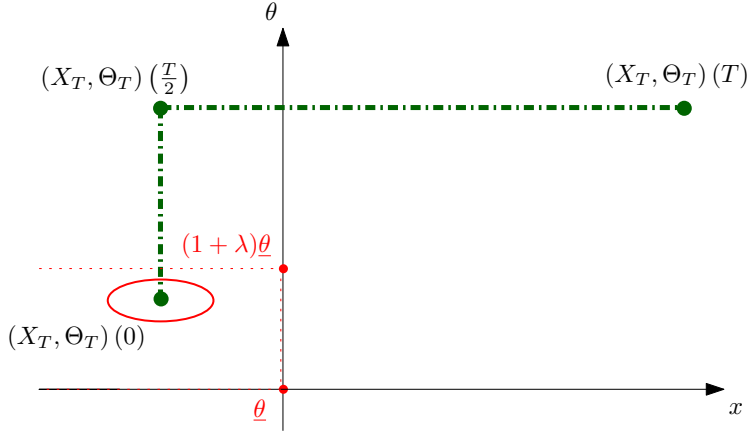}
\end{center}
\caption{Plot of the trajectory in the phase plane $(x,\theta)$. We emphasize that the first part of the trajectory corresponds to a movement towards large traits only in the $\theta$-direction, whereas the second part of the trajectory corresponds to a movement towards large space positions only in the $x$-direction, \textit{at an accelerated spreading rate}. The red dotted line is the support of the initial condition. The red bold line is the initial support of the subsolution $w$ : it sits inside the support of the initial condition.}
\end{figure}

In related works \cite{BCMetal,BHR_acceleration} about the standard local and non-local cane toads equations, some trajectories of the same kind have been introduced.  

Before continuing, we point out that our choice of $\Lambda \leq \lambda \underline \theta / 8$ gives us that
\[
	\cE_{0,\Lambda}\cup \cA_{0,\Lambda}
		\subset (-\infty, 0] \times [\underline\theta, (1+\lambda)\underline\theta].
\]
This implies that $w$ will initially sit below $u$, allowing us to apply the maximum principle.  With this in mind, we state in (a very minimalist fashion) a proposition that will imply our main \cref{thm:acceleration}.

\begin{prop}\label{prop:enough_mass}
Fix $m\in (0,1)$.  There exist positive constants $\lambda_0$, $\Lambda_0$, $T_0$, and $c_0$, such that if $\Lambda \geq \Lambda_0$, $T\geq T_0$, and $c \leq c_0$, then there exists a sub-solution $w$ of the form \eqref{eq:w} satisfying~\eqref{eq:good_convexity} such that 
%
%
%
%functions $w^+$ and $w^-$ satisfying~\eqref{eq:w+} and~\eqref{eq:w-}, respectively, such that
%	\begin{equation}\label{eq:normal_derivative_ordering}
%		|\partial_n w^+| \geq |\partial_n w^-| ~~~~ \text{ on } \partial \mathcal{E}_{t,\Lambda}.
%	\end{equation}
%	
%there exists $\epsilon_{R} \in (0,1/2)$, which tends to $0$ as $R$ tends to infinity, such that
\begin{equation}\label{eq:lower_bound}
\forall (x,\theta) \in \mathcal{E}_{T, \frac{\Lambda}{2}}, \qquad w^+(T,x,\theta) \geq m.
\end{equation}
%	for all $x$ and $\theta$ such that
%	\[
%		(x - X_T(\beta T))^2/\Theta_T(\beta T) + (\theta - \Theta_T(\beta T))^2 \leq R/16.
%	\]
In addition, all constants $\lambda_0$, $\Lambda_0$, $c_0$, and $T_0$ are bounded for $m$ bounded away from $1$, and $\lambda_0$ depends on $\Lambda$.
\end{prop}

We will prove this proposition in \cref{sec:lemma}. Let us now conclude the proof of \cref{thm:acceleration}.

We simply write, as a consequence of the comparison principle, the following:
\begin{equation*}
\begin{array}{ll}
\displaystyle \liminf_{T\to\infty} \; \frac{\max\left\{x : \exists \theta \in \Theta, u(T,x,\theta) \geq m\right\}}{T^{3/2}}
		&\geq \displaystyle \liminf_{T\to\infty} \; \frac{\max\left\{x : \exists \theta \in \Theta, w^+(T,x,\theta) \geq m\right\}}{T^{3/2}} \medskip \\
		&\geq  \displaystyle \liminf_{T\to\infty} \; \frac{X_T^{(2)}(T)}{T^{3/2}}\medskip\\
		&=  \displaystyle \liminf_{T\to\infty} \; \left(\frac{ \frac{c}{2} T^2}{T^{3/2}\left(\sqrt{\frac{c T}{2} + \left( 1 + \frac{3\lambda}{4} \right) \underline\theta}\right)} - \frac{\sqrt{\left(1+\frac{3\lambda}{4}\right) \underline \theta}}{ T^{3/2}}\right) \\
		&=  \displaystyle \sqrt\frac{c}{2}.
\end{array}
\end{equation*}
After calling the latter constant $\gamma$, this finishes the proof.

\begin{rem}
To find optimal constants, one could write a general framework for the trajectories such as for the cane toads equation in \cite{BHR_acceleration} and try to optimize afterwards. However, in this paper, the suboptimal trajectories are not the only reason for sub-optimality.  One would have to quantify all the constants in various lemmas in this paper, that are not strongly related to the definition of the trajectories themselves.  Since this is not our main interest, we opt for the simplest possible presentation.
% As this does not seem very easy to do, we decided not to carry the optimality issue further.
% with any trajectories satisfying 
%	Let $X_T(t)$ and $\Theta_T(t)$ be trajectories on $[\alpha T, \beta T]$ for some $\alpha$ and $\beta$ in $[0,1]$.  Suppose that there exists constants $\delta_X$ and $\delta_\Theta$ such that $X_T$ and $\Theta_T$ satisfy
%	\begin{equation}\label{eq:conditions}
%		\begin{cases}
%			\Theta_T(t) \geq H/2,\\
%			\frac{\dot X_T}{\sqrt{\Theta_T}} = \delta_X,\\
%			\dot\Theta_T = \delta_\Theta,\\
%			 (BOUNDED)etc.
%		\end{cases}
%	\end{equation}
\end{rem}
%
%\begin{rem}
%depending only on $\underline \theta$, $\alpha$, and $r - 2\alpha$ and ???, 
%\end{rem}
%

\section{Building the sub-solution $w$ along trajectories}\label{sec:lemma}

In this section, we will prove \cref{prop:enough_mass}, by working on the two components of the trajectories separately. 
%
%{\color{red}First, we change variables to ones more suited to the second order operator $\theta \partial_x^2 + \partial_\theta^2$.  Secondly, we identify and analyze the steady states for the sets $\mathcal{E}_{t,\Lambda}$ and $A_t$ in the limit $H \to \infty$.  Since, for carefully chosen initial data, $w^+$ converges to its steady state, this is where we show that $w^+$ will be sufficiently large.  Lastly, we show that these steady states have ordered normal derivatives by analyzing their behavior as $R$ tends to infinity.}
Since we want to build a sub-solution along a trajectory $(X_T,\Theta_T)$, we work in the moving coordinates introduced in~\cite{BHR_acceleration})
\begin{equation*}
y = \frac{x - X_T}{\sqrt \Theta_T}, \qquad \eta = \theta - \Theta_T.
\end{equation*}
The moving ellipse and annulus $\mathcal{E}_{t,\Lambda}$ and $\mathcal{A}_{t,\Lambda}$ are respectively transformed into the following stationary circle and annulus $\mathcal{E}_{\Lambda}$ and $\mathcal{A}_{\Lambda}$ :
\begin{equation*}
\mathcal{E}_{\Lambda} = \left\{ (y,\eta) : y^2 + \eta^2 \leq \Lambda^2 \right\}, \qquad \mathcal{A}_{\Lambda} = \left\{ (y,\eta) : \Lambda^2 \leq y^2 + \eta^2 \leq 4 \Lambda^2 \right\}. 
\end{equation*}
The Dirichlet boundary condition that we enforce on $\partial \mathcal{E}_{t,\Lambda}$ pushes us to make the following change of unknown functions: 

\begin{equation*}
w^+(t, x,\theta) = \alpha + v^+\left( t, y , \eta \right), \quad w^-(t, x, \theta) = \alpha - v^- \left( t, y , \eta \right).
\end{equation*}

\noindent We see directly that $v^+$ and $v^-$ solve respectively, for $t \in [0,T]$, 
\begin{equation}\label{eq:v+}
\begin{cases}
v^+_t - c_1(y,t) v^+_y -  c_2(t) v^+_\eta = d(\eta,t) v^+_{yy} + v^+_{\eta\eta} + f_r(\alpha + v^+), \qquad & (y,\eta) \in \mathcal{E}_{\Lambda} , \medskip\\
v^+(t,y,\eta) = 0,&(y,\eta) \in \partial \mathcal{E}_{\Lambda} .
\end{cases}
\end{equation}
\begin{equation}\label{eq:v-}
\begin{cases}
	v^-_t - c_1(y,t) v^-_y - c_2(t) v^-_\eta = d(\eta,t) v^-_{yy} +v^-_{\eta\eta} - f_r(\alpha - v^-),\qquad & (y,\eta) \in \mathcal{A}_{\Lambda} , \medskip\\
	v^-(t,y,\eta) = 0, &(y,\eta) \in \partial \mathcal{E}_{\Lambda} ,\medskip\\
	v^-(t,y,\eta) = \alpha, &(y,\eta) \in \partial \mathcal{A}_{\Lambda}  \backslash \partial \mathcal{E}_{\Lambda} .
\end{cases}
\end{equation}
where we have defined
\begin{equation*}
c_{1}(y,t) := \frac{\dot X_T}{\sqrt \Theta_T} + \frac{y \dot \Theta_T}{2 \Theta_T}, \qquad c_2(t) := \dot\Theta_T, \qquad d(\eta,t) := 1 + \frac{\eta}{\Theta_T}.
\end{equation*}

The maximum principle assures us that $0 \leq v^-\leq \alpha$ holds everywhere in time and space as long as the initial conditions of $v^-$ satisfy the same bound. Similarly, The maximum principle assures us that $0 \leq v^+\leq r- \alpha$ holds everywhere in time and space as long as the initial conditions of $v^+$ satisfy the same bound.

We now need to fix the initial conditions. It is easy to see that the non-constant coefficients in the equation vanish as $\lambda$ tends to infinity.  Hence, it behooves us to look at the steady state solutions to~\eqref{eq:v+} and~\eqref{eq:v-} where the dependence on $\lambda$ disappears.

We define for $i=1,2$, corresponding to each part of the trajectory,
\begin{equation*}
c_{(i)}^{\infty} := \left(  c_{1,(i)}^{\infty}  , c_{2,(i)}^{\infty} \right).
\end{equation*}
For clarity, we emphasize that $c_{(1)}^\infty$ corresponds to the limit of the functions $c_1$ and $c_2$ on the interval $[0,T/2]$ and $c_{(2)}^\infty$ corresponds to the limit of the functions $c_1$ and $c_2$ on the interval $[T/2,T]$.
 
For any $\Lambda \in \R^+$, we thus look for solutions $\phi_{(i)}^{\pm}$ of
\begin{equation}\label{eq:phi-}
\begin{cases}
- \Delta \phi_{(i)}^{-} - c_{(i)}^{\infty} \cdot \nabla \phi_{(i)}^{-} = - f_{r_{(i)}}\left(\alpha - \phi_{(i)}^{-}\right), \qquad &(y,\eta) \in \mathcal{A}_{\Lambda} ,\medskip\\
\phi_{(i)}^{-} (y,\eta) = 0,&(y,\eta) \in \partial \mathcal{E}_{\Lambda} ,\medskip\\
\phi_{(i)}^{-} (y,\eta) = \alpha, &(y,\eta) \in \partial \mathcal{A}_{\Lambda}  \backslash \partial \mathcal{E}_{\Lambda} .
\end{cases}
\end{equation}
and
\begin{equation}\label{eq:phi+}
\begin{cases}
- \Delta \phi_{(i)}^{+} - c_{(i)}^{\infty} \cdot \nabla \phi_{(i)}^{+} = f_{r_{(i)}}\left(\alpha + \phi_{(i)}^{+}\right), \qquad &(y,\eta) \in \mathcal{E}_{\Lambda},\medskip\\
\phi_{(i)}^{+}(y,\eta) = 0,&(y,\eta) \in \partial \mathcal{E}_{\Lambda},
\end{cases}
\end{equation}
where we choose $r_{(1)} = 1$ and we leave $r_{(2)} \in(\max\{m,2\alpha\},1)$ to be determined.  We now state the lemma regarding these steady solutions that we require.

\begin{lem}\label{lem:steady_states}
There exists positive constants $c_0$ and $\Lambda_0$ such that if $\|c^\infty_{(i)}\| \leq c_0$ and $\Lambda > \Lambda_0$, then there exist non-trivial solutions $\phi_{(i)}^{-}$ and $\phi_{(i)}^{+}$ to~\eqref{eq:phi-} and~\eqref{eq:phi+}, respectively.  Moreover, if $\phi_{(i)}^{+}$ is any steady solution to~\eqref{eq:phi+}, then for any $\epsilon > 0$, there exist $\Lambda_\epsilon$ and $L_\epsilon\leq \Lambda_\epsilon/2$, both depending only on $\epsilon$ such that if $\Lambda > \Lambda_\epsilon$ then
\begin{equation}\label{eq:bulk_lower_bound}
%\forall (y,\eta) \in \mathcal{E}_{\left( 1 - \nu_\eps \right) \Lambda}, \qquad 	
	\phi_{(i)}^+(y,\eta) \geq r_{(i)} - \alpha - \epsilon
\end{equation}
	whenever $\dist((y,\eta), \partial \cE_\Lambda) > L_\epsilon$.
\end{lem}

We postpone the proof of this lemma to \cref{sec:steady}. We now complete the definition of the function $w$ by setting the initial conditions
\begin{equation*}
v_{(1)}^-(0,\cdot) = \phi_{(1)}^-, \qquad v_{(1)}^+(0,\cdot) = \phi_{(1)}^+,
\end{equation*}
and
\begin{equation*}
v_{(2)}^-(T/2,\cdot) = \phi_{(2)}^-, \qquad v_{(2)}^+(0,\cdot) = \phi_{(2)}^+,
\end{equation*}

With $v^\pm_{(i)}$ defined, we may now state a proposition guaranteeing that $w$ will satisfy the condition~\eqref{eq:good_convexity} ensuring that it is a sub-solution.  This proposition will be proved at the end of this section.
\begin{prop}\label{prop:subsol}
There exists positive constants $\Lambda_0$, $T_0$, and $c_0$, 
%depending only on $\underline \theta$, $\alpha$, and $r - 2\alpha$ and ???, 
such that if $\Lambda \geq \Lambda_0$, $T\geq T_0$, and $c \leq c_0$, then $v^+$ and $v^-$, defined above, satisfy
	\begin{equation}\label{eq:normal_derivative_ordering}
		\vert \partial_n v^+_{(i)} \vert \geq \vert \partial_n v^-_{(i)} \vert
	\end{equation}
	on $\mathcal{E}_{\Lambda}$.
%Thus, the (now fully) defined function $w$ is a sub-solution to $u$ on $\R^+ \times \R \times \Theta$. 
\end{prop}
%This proposition, and in particular the crucial fact that the normal derivatives are ordered, will be proved at the end of this section. 

With the initial conditions we have fixed for $v^\pm$, it is clear that as $\Lambda$ tends to infinity, $v^\pm$ should tend to $\phi^\pm$ uniformly in the $C^1$ topology. The following lemma ensures that this convergence holds and is uniform in $T$.
\begin{lem}\label{lem:convergence_steady}
	Fix $\epsilon > 0$.  There exists $\Lambda_\epsilon$ and $\lambda_{\epsilon}$ such that if $\Lambda \geq \Lambda_\epsilon$ and $\lambda \geq \lambda_{\epsilon}$ then
	\[
		\|v_{(i)}^\pm(t,\cdot) - \phi_{(i)}^\pm(\cdot)\|_{C^1} \leq \epsilon.
	\]
	Moreover, this convergence holds uniformly in $t$.
\end{lem}

%We point out that the dependence of $H_{R,\epsilon}$ on $R$ is unnecessary, but for the sake of simplicity, we only require this weaker result. 
We postpone the proof of this Lemma until the dedicated \cref{sec:convergence_steady}. We now show how the combination of \cref{lem:steady_states} and \cref{lem:convergence_steady} finally implies \cref{prop:enough_mass}. After changing variables it is equivalent to proving the same result for $v^+$. This is exactly now that we are going to travel along the trajectory $\left( X_T ,\Theta_T\right)$ and play with parameters.

%
% Hence we need only show that the normal derivatives of $v^+$ and $v^-$ are ordered as in~\eqref{eq:normal_derivative_ordering} in order to finish the proof of \cref{lem:normal_derivative}.
%
We start with the first part of the trajectory.  Fix a positive constant $\epsilon < (1 - m)/3$.  The discussion above implies that we may find $\Lambda_0,T_0$ and $\delta_0$ such that if $\Lambda \geq \Lambda_0$, $T \geq T_0$, and $c \leq c_0$, then we have that $v^+$ gives a subsolution to $u$ such that 
\begin{equation}\label{eq:first_lower_bound}
u\left(\frac{T}{2}, \cdot\right) \geq \alpha + v^+ \left( \frac{T}{2} , \cdot  \right) \geq \phi_{(1)}^+ \left( \cdot \right) - \eps + \alpha \geq 1 - 2 \eps
\end{equation}
on the set $\mathcal{E}_{\Lambda/2}$. This concludes our sliding over the first part of the trajectory.

First, by increasing $\Lambda_0$ if necessary, we may find $r_{(2)} := r \geq \epsilon + m$ in the interval $(2\alpha,  1 - 2 \eps  )$. By the maximum principle any solution to~\eqref{eq:w+} will be less than or equal to $r$. 

With the choice $\Lambda' = \Lambda/2$, we find that, by increasing $T_0$ and $\Lambda_0$, if necessary, if $T \geq T_0$ and $\Lambda' \geq \Lambda_0$, then $w$ is a sub-solution to $u$ such that  
\begin{equation*}
	u(T,\cdot)
		\geq w_{(2)}^+(T , \cdot)
		= v_{(2)}^+(T,\cdot) +\alpha
		\geq \phi_{(2)}^+ \left(\cdot \right) - \eps + \alpha
		\geq m.
\end{equation*}
on the set $\mathcal{E}_{\frac{\Lambda}{2}}$.  We point out that~\eqref{eq:first_lower_bound} is crucial to guarantee that $w$ is a sub-solution to $u$ on the interval $[T/2, T]$.

We finish this section by reducing \cref{prop:subsol} to the following lemma.
\begin{lem}\label{lem:der_steady_states}
For $\Lambda$ sufficiently large,
\begin{equation}\label{eq:phi_normal}
\vert \partial_n \phi_{(i)}^+( \Lambda \hat v) \vert > \vert \partial_n \phi_{(i)}^-( \Lambda \hat{v}) \vert
\end{equation}
holds uniformly in $\hat{v}  \in \mathbb{S}^1(\R^2)$.
\end{lem}
We postpone the proof of this lemma to \cref{sec:steady} below.

\begin{proof}[{\bf Proof of \cref{prop:subsol}}]
%To prove \cref{prop:subsol}, we need to prove the crucial fact that the normal derivatives are ordered on $\partial \mathcal{E}_{\Lambda}$. 
Fix any unit vector $\hat v \in \mathbb{S}^1(\R^2)$ and look at any point $\Lambda \hat{v}$ on the boundary of $\mathcal{E}_{\Lambda}$.  Then~\eqref{eq:normal_derivative_ordering} follows from \cref{lem:convergence_steady} and \cref{lem:der_steady_states}.
\end{proof}

\section{Lemmas about the steady solutions}\label{sec:steady}

In this section, we prove \cref{lem:steady_states} and \cref{lem:der_steady_states}, which give the desired behavior for the steady states $\phi^\pm_{(i)}$.

\begin{proof}[{\bf Proof of \cref{lem:steady_states}}]
	The existence of $\phi^\pm$ is an easy result of a fixed point argument using the principle Dirichlet eigenfunction as a sub-solution to obtain non-triviality.  As such, we omit it here.
	
%	We show uniqueness for $\phi^+$, though the same argument may be applied to $\phi^-$.  To show uniqueness, suppose that we have two positive solutions $\phi_1$ and $\phi_2$ to~\eqref{eq:phi+}.  By elliptic regularity, we may find $M$ large enough that $M\phi_1 \geq \phi_2$.  Define
%	\[
%		M_0 = \inf\{M : M \phi_1 \geq \phi_2\},
%	\]
%	and define $\psi = M_0 \phi_1$.  We will show that $M_0 \leq 1$.  By the Hopf maximum principle, there exists a point in the interior of $E$, $(y_0,\eta_0)$, such that $\psi(y_0,\eta_0) = \phi_2(y_0,\eta_0)$.  Hence we have that $\nabla \psi(y_0, \eta_0) = \nabla \phi_2(y_0,\eta_0)$ and we have that $\Delta \psi(y_0,\eta_0) \geq \phi_2(y_0,\eta_0)$.  Hence we have that
%	\[
%		0 \geq - (\delta_X, \delta_\Theta) \cdot \nabla (\psi - \phi_2)
%			- \Delta(\psi - \phi_2)
%			= M_0 f_r(\psi/M_0) - f_r(\phi_2).
%	\]
%	at $(y_0,\eta_0)$.
	
	We now show the lower bound for $\phi^+$.  We first show a preliminary lower bound for $\phi^+$ that we will bootstrap to the lower bound~\eqref{eq:lower_bound}. Let $L$ be a constant yet to be determined. Let $(y_L,\eta_L)$ be the location of the maximum of
\begin{equation*}
		e^{- \frac12 c^{\infty} \cdot (y,\eta)} \cos \left( \frac{\pi y}{ 2L}\right) \cos\left( \frac{\pi \eta}{ 2L}\right).
\end{equation*}
%We point out that $(y_m,\eta_m)$ is bounded as $L$ tends to infinity.
Now fix any point $(y_0,\eta_0)$ to be determined and define
	\begin{equation}\label{eq:tilde_psi}
		\tilde \psi(y,\eta) = e^{- \frac12 c^{\infty} \cdot (y+ y_L -y_0,\eta+ \eta_L -\eta_0)} \cos\left( \frac{\pi}{2L} (y+y_L - y_0)\right)\cos\left( \frac{\pi}{2L} (\eta + \eta_L - \eta_0)\right),
	\end{equation}
Notice that, by construction, the maximum of $\tilde \psi$ occurs at $( y_0 ,\eta_0 )$.
%If $c_0$ is sufficiently small and using that $r > 2\alpha$,
%we may fix a parameter
%	\[
%		a \in \left( \frac{\Vert c^\infty \Vert^2}{4c_r \alpha}, r - \alpha\right).
%	\]
	Define
	\begin{equation}
		\psi = \frac{\gamma}{\tilde\psi( y_0,\eta_0)} \tilde\psi,
	\end{equation}
	where $\gamma$ is a small parameter to be chosen later.
It is easy to check that
\begin{equation*}
\begin{split}
- \Delta \psi - c^\infty \cdot \nabla \psi - f_r(\alpha + \psi) \leq \left(\frac{\Vert c^\infty \Vert^2}{4} + \frac{\pi^2}{2 L^2} - \frac{c_r \alpha (r-\alpha)}{2} \right) \psi \leq 0,
\end{split}
\end{equation*}
on the set
\[
	\mathcal{C} = \left\lbrace (y,\eta) \, \vert \, \vert y - (y_0 - y_L)\vert = L \text{ or } \vert \eta - (\eta_0 - \eta_L) \vert  = L \right\rbrace.
\]
	The first inequality follows by a simple computation and by choosing $\gamma$ sufficiently small, depending only on $r$ and $f$, such that
	\[
		\frac{c_r \alpha (r-\alpha)}{2} u \leq f_r(\alpha+u)
	\]
	for all $u \in [0,\gamma]$.
	The last inequality follows by decreasing $c_0$ and increasing $L$, if necessary.
	
	Hence, $\psi$ is sub-solution to $\phi^+$ which satisfies Dirichlet boundary conditions on the square $\mathcal{C}$.
This implies that as long as $\mathcal{C} \subset \mathcal{E}_{\Lambda}$, 
%	\[
%		|(y_0,\eta_0)|^2
%			\leq R/4 - 2L^2 - |(y_m,\eta_m)|^2.
%	\]
%	hold, 
we have $\psi \leq \phi^+$ on $\mathcal{C}$.  We point out that a weak condition for $\mathcal{C} \subset \mathcal{E}_\Lambda$ is that
\begin{equation}\label{eq:weak_lower_bound}
	\dist((y_0,\eta_0), \partial\cE_{\Lambda}) \geq 2L + |y_L| + |\eta_L|.
\end{equation}
This, in turn, implies that
\begin{equation*}
\phi^+(y_0,\eta_0) \geq \gamma,
\end{equation*}
and this may be valid as soon as $(y_0,\eta_0)$ satisfy~\eqref{eq:weak_lower_bound}.
	
	Now we claim that $\phi^+$ tends to $r - \alpha$ locally uniformly in the set $\mathcal{E}_{\Lambda - L}$ as $\Lambda$ and $L$ tend to $\infty$. To prove this, we argue by contradiction.  If this is not true, then we may find a positive $\epsilon > 0$, sequences $L_k \leq \Lambda_k/2$ tending to infinity, and a sequence of points $(y_k,\eta_k) \in  \mathcal{E}_{\Lambda_k - L_k}$ such that $\phi^+(y_k,\eta_k) \leq r - \alpha - \epsilon$.
	
	Let $\phi_k^+(y,\eta) = \phi^+(y + y_k, \eta+\eta_k)$.  Elliptic regularity ensures that $\phi_k^+$ converges locally uniformly in $\mathcal{C}^2$ to a function $\phi_\infty^+$.  By assumption, we have that $\phi_\infty^+(0,0) \leq r - \alpha - \epsilon$.  By our work above, we have that $\phi_\infty^+ \geq \gamma$ everywhere.  By the maximum principle, we have that $\phi_\infty^+ \leq r - \alpha$ everywhere.  Finally, we have that $\phi_\infty^+$ satisfies~\eqref{eq:phi+}.  We may assume that $(0,0)$ is a positive global local minimum for $\phi_\infty$ since, if not, we may simply re-center the equation again, taking limits if necessary.  It is apparent, however, that a solution to~\eqref{eq:phi+} can only have a minimum at the zeros of $f_r(\alpha + \cdot)$, which are $0$ and $r-\alpha$.  Thus, we arrive at a contradiction.  Hence, it must be that $\phi^+$ tends uniformly to $r - \alpha$, finishing the proof of the claim.
\end{proof}

We are now in a position to prove \cref{lem:der_steady_states}, showing the ordering of the normal derivatives for the steady states $\phi^\pm_{(i)}$.

\begin{proof}[{\bf Proof of \cref{lem:der_steady_states}}]

%Fix any unit vector $\hat v$ and look at any point $(\sqrt{R}/2) \hat{v}$ on the boundary of $E$.  Then~\eqref{eq:normal_derivative_ordering} follows from \cref{lem:convergence_steady} and from showing that
%\begin{equation}\label{eq:phi_normal}
%	|\partial_n \phi^+( (\sqrt{R}/2) \hat v)|
%		> |\partial_n \phi^-( (\sqrt{R}/2) \hat{v})|
%\end{equation}
%holds uniformly for all $\hat{v}$ for all $R$ sufficiently large.
%
%In view of our work above, we need only prove~\eqref{eq:phi_normal} to conclude the proof of \cref{lem:normal_derivative}.  Hence, w

We assume the contrary.  Suppose there is a sequence of $\hat{v}_k$ and $\Lambda_k$, with $\Lambda_k$ tending to infinity, such that
\begin{equation}\label{eq:normal_contradiction1}
	\left\vert\partial_n \phi^+\left( \Lambda_k \hat{v}_k \right) \right\vert \leq \left\vert \partial_n \phi^-\left(\Lambda_k\hat{v}_k \right) \right\vert + o_{k \to + \infty}(1),
\end{equation}
By the compactness of the unit circle, $\hat{v}_k$ converges (up to a subsequence) to some unit vector $\hat{v}_\infty$.

Define $(y_k,\eta_k) = \Lambda_k \hat{v}_k$.  We center the solutions on it.  In other words, we define
\begin{equation*}
z_k^-(y,\eta) = \phi^-(y_k + y,  \eta_k + \eta), \qquad z_k^+(y,\eta) = \phi^+(y_k - y, \eta_k - \eta).
\end{equation*}
The new function $z_k^{+}$ satisfies
\begin{equation}\label{eq:z+}
\begin{cases}
- \Delta z_k^+ + c^\infty \cdot \nabla z_k^+ = f_r(\alpha + z_k^+), \qquad &(y-y_k,\eta-\eta_k) \in \mathcal{E}_{\Lambda_k},\medskip\\
z_k^+(y,\eta) = 0,&(y-y_k,\eta - \eta_k) \in \partial \mathcal{E}_{\Lambda_k},
\end{cases}
\end{equation}
We point out that the condition 
\begin{equation*}
\begin{array}{lcl}
(y-y_k,\eta-\eta_k) \in E_{\Lambda_k} & \Longleftrightarrow & \Vert (y,\eta) \Vert^2 - 2 (y,\eta) \cdot  (\Lambda_k \hat{v}_k)  +  \Lambda_k^2 \leq  \Lambda_k^2,\medskip\\
& \Longleftrightarrow & \Vert (y,\eta) \Vert^2 \leq 2 \Lambda_k (y,\eta) \cdot \hat{v}_k .
\end{array}
\end{equation*}
The new function $z_k^{-}$ satisfies
\begin{equation}\label{eq:z-}
\begin{cases}
- \Delta z_k^- - c^\infty \cdot \nabla z_k^- = - f_r(\alpha - z_k^-), \qquad &(y+y_k,\eta+\eta_k) \in \mathcal{A}_{\Lambda_k} ,\medskip\\
z_k^-(y,\eta) = 0,&(y+y_k,\eta+\eta_k) \in \partial \mathcal{E}_{\Lambda_k} ,\medskip\\
z_k^-(y,\eta) = \alpha, &(y+y_k,\eta+\eta_k) \in \partial \mathcal{A}_{\Lambda_k}  \backslash \partial \mathcal{E}_{\Lambda_k} .
\end{cases}
\end{equation}
We point out that the condition 
\begin{equation*}
\begin{array}{lcl}
(y+y_k,\eta+\eta_k) \in A_{\Lambda_k}  & \Longleftrightarrow & \Lambda_k^2 \leq \Vert (y,\eta) \Vert^2 + 2 (y,\eta) \cdot ( \Lambda_k \hat{v}_k ) +  \Lambda_k^2 \leq  4 \Lambda_k^2,\medskip\\
& \Longleftrightarrow & 0 \leq \Vert (y,\eta) \Vert^2 + 2 \Lambda_k  (y,\eta) \cdot \hat{v}_k  \leq  3 \Lambda_k^2,\medskip\\
\end{array}
\end{equation*}

By elliptic regularity theory, we may take local uniform limits to obtain $z^-$ and $z^+$ as the limits of $z^-_k$ and $z^+_k$, respectively, which solve
\begin{equation}\label{eq:z+}
\begin{cases}
- \Delta z^+ + c^\infty \cdot \nabla z^+ = f_r(\alpha + z^+), \qquad &(y,\eta) \in P,\medskip\\
z^+(y,\eta) = 0,&(y,\eta) \in \partial P,
\end{cases}
\end{equation}
and
\begin{equation}\label{eq:z-}
\begin{cases}
- \Delta z^- - c^\infty \cdot \nabla z^- = - f_r(\alpha - z^-), \qquad &(y,\eta) \in P,\medskip\\
z^-(y,\eta) = 0,&(y,\eta) \in \partial P,\\
\end{cases}
\end{equation}
where we have defined the half-plane 
\begin{equation*}
	P \stackrel{\rm def}{=} \left\{ (y',\eta') : (y',\eta') \cdot \hat{v}_\infty > 0\right\}.
\end{equation*}

Due to \cref{lem:steady_states}, we conclude the following about $z^\pm$.  In addition, $z^\pm(y,\eta)$ is uniformly positive away from this boundary.  Furthermore, $z^+(y,\eta)$ converges to $r - \alpha$ and $z^-(y,\eta)$ converges to $\alpha$ as the distance between $(y,\eta)$ and $\partial P$ tends to infinity.  

Finally, we point out that~\eqref{eq:normal_contradiction1} implies that $\partial_n z^+(0,0) \leq \partial_n z^-(0,0)$, where we have removed the absolute values since the sign of both quantities is clearly positive.

This contradicts the following lemma, which we prove in the following section as it requires the results from \cref{lem:uniqueness}.
\begin{lem}\label{lem:centered_normal_derivative}
There exists $\delta > 0$ such that if $\Vert c^\infty \Vert \leq \delta$ then $\partial_n z^+(0,0) > \partial_n z^-(0,0)$. 
%
%
%%	\[
%%		\left| \frac{\dot X_T}{\sqrt \Theta_T}\right| + \left| \dot \Theta_T\right|
%%			\leq \delta
%%	\]
%	then it follows that .
\end{lem}
%\noindent We prove this lemma in the following section since its proof, though involving well-known techniques, is involved.
%
Since we have reached a contradiction, it must be that~\eqref{eq:phi_normal} holds for all $\Lambda$ sufficiently large.  This finishes the proof of \cref{lem:der_steady_states}.
\end{proof}

\section{Convergence to the steady states: Proof of \cref{lem:convergence_steady}}\label{sec:convergence_steady}

Before we prove \cref{lem:convergence_steady}, we prove a few useful lemmas.  The first of these is the uniqueness of solutions to the related steady half-plane problem.  We require this to prove \cref{lem:centered_normal_derivative}, stated above.

We also note that \cref{lem:uniqueness} with $c^\infty = 0$ is well-known.  We include the proof of the case with constant $c^\infty$, which is virtually unchanged, for completeness.

\begin{lem}\label{lem:uniqueness}
Let $(y_b,\eta_b) \in \R^2$ and $\hat{v} \in \mathbb{S}^1(\R^2)$.  Define 
\begin{equation*}
P = \{ (y,\eta): 0 \leq (y+y_b, \eta + \eta_b)\cdot \hat{v}\}.
\end{equation*}
For any choice of $c^\infty$, there is at most one solution to
\begin{equation}\label{eq:zlem}
\begin{cases}
- \Delta z^+ + c^\infty \cdot \nabla z^+ = f_r(\alpha + z^+), \qquad &(y,\eta) \in P,\medskip\\
z^+(y,\eta) = 0,&(y,\eta) \in \partial P,
\end{cases}
\end{equation}
which satisfies Dirichlet boundary conditions on $\partial P$ and which tends uniformly to $r - \alpha$ as $\dist((y,\eta), \partial P)$ tends to infinity.

In addition, for any choice of $c^\infty$, there is at most one solution to
\begin{equation}\label{eq:zlem-}
\begin{cases}
- \Delta z^- + c^\infty \cdot \nabla z^- = - f_r(\alpha - z^-), \qquad &(y,\eta) \in P,\medskip\\
z^-(y,\eta) = 0,&(y,\eta) \in \partial P,
\end{cases}
\end{equation}
which satisfies Dirichlet boundary conditions on $\partial P$ and which tends uniformly to $\alpha$ as $\dist((y,\eta), \partial P)$ tends to infinity.
\end{lem}
\begin{proof}[{\bf Proof of \cref{lem:uniqueness}}]
We show the argument for $z^+$ but the argument for $z^-$ is identical.  By rotation and translation, we may assume that $\hat{v} = (1,0)$ and that $(y_b,\eta_b) = (0,0)$.  We will use a sliding method \cite{BerestyckiNirembergSliding}.

We proceed by contradiction and assume that $z_1$ and $z_2$ solve~\eqref{eq:zlem} and satisfy $z_2(y_0,\eta_0) < z_1(y_0,\eta_0)$ for some point $(y_0,\eta_0)$.

Fix $u_0 \in (0, r-\alpha)$ large enough such that
	\[
		\tilde f(u) := u(\alpha + u)(r - \alpha - u)
	\]
	is monotonically decreasing for $u \geq u_0$.  Since $z_1$ and $z_2$ tend uniformly to $r-\alpha$ as the$y$ tends to infinity, we may find $L$ such that if $y > L$, then $z_i(\tau, y, \eta) \geq u_0$ for $i = 1,2$.

We define
\begin{equation*}
	z_{2}^\tau(y,\eta) := z_{2}(y + \tau, \eta).
\end{equation*}	
Since $z_{2}$ converges locally uniformly to $r-\alpha$ as $y$ tends to infinity, we may find $\tau_0$ sufficiently large that $z_{2}^{\tau_0} > z_{1}$ on the set
	\[
		P_L = \left\{(y,\eta) \in P : 0 < y < L\right\}.
	\]
	
	We claim that $z_{2}^{\tau_0} \geq z_{1}$ on $P$.  If not, let $\psi = z_{2}^{\tau_0} - z_{1}$ and let $G = \{ \psi < 0\}$.  It is easy to see that $G \subset P\setminus P_L$, by construction.  Also, it is easy to check that $\psi$ satisfies
\begin{equation*}
- \Delta \psi - c^\infty \cdot \nabla \psi = \underbrace{\left(\frac{\tilde f(z_{2}^{\tau_0}) - \tilde f(z_{1})}{z^{\tau_0}_{2} - z_{1}}\right)}_{:=\hat f(y,\eta)} \psi,
\end{equation*}
	on the set $G$.  Since $\psi$ is zero on $\partial G$ and since $\hat f\leq 0$ by our assumptions on $L$ and $u_0$, the maximum principle ensures that $\psi \equiv 0$ on $G$, which is a contradiction.  Hence, we have that $z^{\tau_0}_{2} \geq z_{1}$ on $P$.
	
	Let 
	\[
		\overline \tau = \inf \{ \tau: z^{\tau}_{2} \geq z_{1} \text{ on } P \},
	\]
	and notice that $\overline \tau \leq \tau_0$.  By continuity, it follows that $z_2^{\overline\tau} \geq z_1$.  We claim that $\overline \tau = 0$, and we proceed by supposing the opposite, that is $\overline \tau > 0$. There are two cases: either $\inf_{P_L} \left(  z^{\overline\tau}_2 - z_1 \right) > 0$, or $\inf_{P_L} \left( z^{\overline\tau}_2 - z_1\right) = 0$.
	
	We handle the former case first.  In this case, elliptic regularity ensures that $z_1,z_2$ are uniformly Lipschitz continuous, and that we may then decrease $\overline \tau$ to a positive constant $\overline \tau' < \overline \tau$ and preserve the fact that $z^{\overline\tau'}_2 \geq z_1$ on $P_L$.  Arguing as above, this implies that $z^{\overline\tau'}_2 \geq z_1$ on $P$.  This is contradicts the definition of $\overline \tau$.
	
	Now we handle the latter case.  Assume that $\inf_{P_L} \left(z_1^{\overline\tau} - z_2\right) = 0$.  Then there exists a sequence of points $(y_k,\eta_k) \in P_L$ such that $z_1^{\overline\tau}(y_k,\eta_k) - z_2(y_k,\eta_k)$ tends to zero.  Defining
\begin{equation*}
\psi_{1,k} (y, \eta) = z_1(y+y_k,\eta+\eta_k), \qquad \psi_{2,k}(y,\eta) = z_2^{\overline\tau}(y+y_k,\eta+\eta_k).
\end{equation*}
	By elliptic regularity, we may, taking a subsequence if necessary, find $\overline\psi_i$ which are the local limits of $\psi_{i,k}$ as $k$ tends to infinity.  In addition, $y_k$ converges to some $\overline y$  since $(y_k,\eta_k) \in P_L$ implies that $y_k$ is bounded.  Hence we have that $\overline\psi_i$ satisfy~\eqref{eq:zlem} on $P_L$ and satisfy that
	\[
		\overline\psi_1(0,0) = \overline \psi_2(0,0).
	\]
	By the maximum principle, this implies that $\overline \psi_1 \equiv \overline \psi_2$ on $P_L$.  On the other hand, since $\overline \tau > 0$, we have that, $\psi_2(y,\eta)$ is positive for $y = - \overline y$ while $\psi_1(y,\eta)$ is zero for $y = -\overline y$.  This is a contradiction.
	
	Hence, it must hold that $\overline \tau = 0$.  By the definition of $\overline \tau$, this implies that $z_2 = z_2^{\tau  =0} \geq z_1$.  This contradicts our original assumption that $z_2(y_0,\eta_0) < z_1(y_0,\eta_0)$.  This finishes our proof.

\end{proof}

With \cref{lem:uniqueness} in hand, we may now conclude the proof \cref{lem:centered_normal_derivative}.

\begin{proof}[{\bf Proof of \cref{lem:centered_normal_derivative}}]
Suppose that this is not true.  Then there exist a sequence $\delta_k$ tending to $0$ as $k$ goes to infinity and a sequence $c^\infty_{k}$ such that
\begin{equation*}
	\Vert c_k^\infty \Vert \leq \delta_k
\end{equation*}
and functions $z_k^+$ and $z_k^-$ solving \eqref{eq:z+} and \eqref{eq:z-} respectively which satisfy
\begin{equation}\label{eq:normal_contradiction}
	\partial_n z_k^+(0,0) \leq \partial_n z_k^-(0,0).
\end{equation}
By elliptic regularity theory, we may take local uniform limits to obtain that $z_k^+$ and $z_k^-$ converge to $z_{\infty}^+$ and $z_{\infty}^-$, respectively, which are the unique solutions to~\eqref{eq:z+} and~\eqref{eq:z-} with $c^\infty = 0$.

Rotating if necessary, we may now assume that $\hat{v} = (1,0)$.  Using \cref{lem:uniqueness}, it is easy to see that $z_{\infty}^+$ and $z_{\infty}^-$ depend only on the $y$-variable. It is easily checked that
\begin{equation*}
(\forall u > 0) \qquad f_r(\alpha + u) \leq -f_r(\alpha - u).
\end{equation*}

Hence, we may use a sliding technique \cite{BerestyckiNirembergSliding} (see also \cref{lem:uniqueness} below) to show that $z_\infty^+ > z_\infty^-$. Finally, the Hopf Lemma implies that $\partial_n z_{\infty}^+ > \partial_n z_{\infty}^-$, which contradicts~\eqref{eq:normal_contradiction}.  This finishes the proof.
\end{proof}

The following lemma addresses the issue of uniqueness for the steady problem~\eqref{eq:phi+} on the bounded ellipse $\cE_\Lambda$.  While we expect uniqueness to hold in this setting, we are unable to prove it because $f_r$ need not satisfy the condition that $f_r(\alpha+ u)/u$ that is usually used for uniqueness results of this type.  As such, we settle for a type of asymptotic uniqueness, stated below.

\begin{lem}\label{lem:close_steady}
For any $\epsilon > 0$, there exists $\Lambda_\epsilon$ such that if $\Lambda \geq \Lambda_\epsilon$ any two positive solutions to~\eqref{eq:phi+} must be $\epsilon$-close in the $L^\infty$ norm.
\end{lem}
\begin{proof}[{\bf Proof of \cref{lem:close_steady}}]
	Suppose the opposite, that there is a positive constant $\epsilon$, a sequence of $\Lambda_k$ tending to infinity, and two sequences of functions $\phi_{1,k}$ and $\phi_{2,k}$ which both solve~\eqref{eq:phi+} such that $\|\phi_{1,k} - \phi_{2,k}\|_\infty \geq \epsilon$.  Then we may find a sequence of points $(y_k,\eta_k)$ such that $\phi_{1,k}(y_k,\eta_k) \geq \phi_{2,k}(y_k,\eta_k) + \epsilon$.  We may select $(y_k,\eta_k)$ to be a local maximum of $\phi_{1,k} - \phi_{2,k}$.
		
	Now define functions
	\[
		\tilde \phi_{i,k} (y,\eta) = \phi_{i,k}(y + y_k, \eta +\eta_k) ~~\text{ for } i = 1,2.
	\]
	By elliptic regularity, it follows that $\tilde \phi_{i,k}$ converge, along a subsequence, if necessary, locally uniformly to $\overline \phi_{i}$. 
	
	There are two cases to consider here: either $(y_k,\eta_k)$ remains a finite distance from the boundary of $B_{\Lambda_k}$ or it does not.  The latter cannot happen due to \cref{lem:steady_states}.  Hence it follows that $\overline \phi_{i}$ solve~\eqref{eq:phi+} on a half-space
	\[
		P = \left\{(y,\eta): (y + y_\infty,\eta + \eta_\infty) \cdot \hat{v} > 0\right\}
	\]
	for some unit vector $\hat{v}$ and some translation $(y_\infty, \eta_\infty)$, with Dirichlet boundary conditions on $\partial P$.  From \cref{lem:uniqueness}, we know that \eqref{eq:phi+} has a unique positive solution on $P$.  From this, it follows that $\overline \phi_1 = \overline \phi_2$.  This contradicts the fact that $\overline \phi_1(0,0) \geq \overline \phi_2(0,0) + \epsilon$.  This contradiction finishes the proof of the claim.
\end{proof}

%
%
%
%
%
%{\color{red} OF COURSE WE EXPECT UNIQUENESS ABOVE BUT WE ARE UNABLE TO PROVE IT OR FIND A REFERENCE. - should cite some of the sliding things...}

Lastly, we need one more ingredient in order to prove \cref{lem:convergence_steady}.  This is a uniform lower bound to $v^\pm$ to ensure that when taking limits in the arguments, we do not lose positivity of $v^\pm$.

\begin{lem}\label{lem:uniformly_positive}
	There exists $\Lambda_0$ and $\lambda_0$ such that if $\Lambda \geq \Lambda_0$ and $\lambda \geq \lambda_0$, then $v^+$ and $v^-$ are uniformly positive for any subset $K$ of $\mathcal{E}$ and $\mathcal{A}$, respectively.  Moreover, the lower bound may be chosen in such a way to depend only on the distance between the subset $K$ and the boundary $\partial \mathcal{E}$ and $\partial \mathcal{A}$, respectively.
\end{lem}
\begin{proof}[{\bf Proof of \cref{lem:uniformly_positive}}]
This follows easily by using the time dependent eigenvalues of Lemmas 5.1 and 5.2 from~\cite{BHR_acceleration} along with the strategy from \cref{lem:steady_states}.  As such we omit the proof.

%We point out that this proof automatically gives a lower bound $\underline v^\pm$ for $v^\pm$ which is a sub-solution to $v^\pm$.  
\end{proof}

We now finally tackle the convergence result of \cref{lem:convergence_steady}.

\begin{proof}[{\bf Proof of \cref{lem:convergence_steady}}]
	Before we begin, we point out that, by parabolic regularity, it is enough to show that $v^\pm$ converges to $\phi^\pm$ uniformly in $L^\infty$.  In addition, we prove the claim for $v^+$ and $\phi^+$, but the proof is identical for $v^-$ and $\phi^-$.
		
	Fix $\epsilon > 0$.  Choose $\Lambda_\epsilon$ large enough that we may apply \cref{lem:close_steady} to get that all non-trivial solutions to~\eqref{eq:phi+} are within $\epsilon/4$ in the $L^\infty$ norm.
	
	Now we proceed by contradiction. Namely, we assume that there exists $y_k$, $\eta_k$, $\lambda_k$, $T_k$ and $t_k\in [0,T_k]$ such that $\lambda_k$ and $T_k$ tend to infinity and such that
	\[
		|v^+(t_k, y_k, \eta_k) - \phi^+(y_k,\eta_k)| \geq \epsilon.
	\]
	Define $v_k(t,\cdot) = v^+(t+t_k, \cdot)$.
	
	By the compactness of $\mathcal{E}$, we have that $(y_k, \eta_k)$ tends, along a sub-sequence if necessary, to a point $(\overline  y, \overline \eta)$.  In addition, \cref{lem:uniformly_positive} ensures that $v_k$ is uniformly positive on any compact subset of $E$.
	
	There are two cases: either $t_k$ is bounded or not.  If it is bounded, then the choice of initial conditions for $v^+$ along with parabolic regularity, see e.g.~\cite{KrylovHolder}, implies that $v_k$ converges to $\phi^+$.  Hence it must be that $t_k$ tends to infinity.
	
	In this case, it is easy to check that $v^+_k$ falls within $3\epsilon/4$ of a steady state of~\eqref{eq:phi+}.  Indeed, consider the limit $v_\infty^+ (t,y,\eta)$.  First, notice that this must solve the parabolic analogue of~\eqref{eq:phi+}.  Second, notice that, for any $t_0>0$, 
	\[
		w_0(y,\eta) \leq v_\infty(-t_0, y,\eta)
	\]
	for a positive lower bound $w_0$, independent of $t_0$, given by \cref{lem:uniformly_positive}.  Using the linearized equation for $v_\infty$ it is clear that we may find $\underline w_0 \leq w_0$ which is a super-solution to the related elliptic problem~\eqref{eq:phi+}.  Letting $\underline w$ be the solution to the parabolic analogue of~\eqref{eq:phi+} starting from $\underline w_0$ we have that
	\[
		\underline w(t_0,y,\eta) \leq v_\infty(0,y,\eta)
	\]
	by the maximum principle.  By our choice of initial conditions, it must be that $\underline w_\tau > 0$ for all times and for all $y$ and $\eta$.  By the boundedness of $\underline w$, it follows that $\underline w_\tau$ tends uniformly to zero and hence, to a solution of~\cref{eq:phi+}.  We may similarly find an upper bound for $v_\infty(0,y,\eta)$ in terms of a non-trivial solution of~\cref{eq:phi+}.  Hence, choosing $t_0$ large enough and applying\cref{lem:close_steady} with the choice $\epsilon' = \epsilon/8$, we have that $v_\infty$ is within $\epsilon / 4$ of a steady state.
	
%	On the other hand, each of $\underline w$ and $\overline w$ converge uniformly a steady state of~\eqref{eq:phi+} as $t_0$ tends to infinity.  Hence, using our choice of $\Lambda$ above, we have that $v_\infty^+(0,y,\eta)$ must be within $\epsilon/4$ of a steady state to~\eqref{eq:phi+}.
	
	%In addition, by the uniform positivity of $v^+$, we have that $\overline v$ is uniformly positive as well.  Hence, we have that $v^+$ must be a non-trivial steady state to the equation above.
	Applying \cref{lem:close_steady} again implies that $v^+$ and $\phi^+$ are within $\epsilon /2$, using again our choice of $\Lambda$ above.  However, this contradicts our assumption that $| v^+(0, \overline y, \overline \eta) - \phi^+(\overline y, \overline \eta)| \geq \epsilon$, finishing the proof.

\end{proof}

\bibliography{refs}
\bibliographystyle{plain}

\end{document}